\documentclass[12pt,reqno]{article}
mm \textheight=8.9in

\usepackage{amssymb}
\usepackage{amsmath}
\usepackage{amsthm}
\usepackage{pb-diagram}
\usepackage{color}

\newcommand{\br}[3]{{$#1$}$\lower4pt\hbox{$\tp\atop\raise4pt \hbox{$\scriptscriptstyle{#2}$}$} ${$#3$}}
\newcommand{\tw}[3]{{$#1$}${\,\scriptscriptstyle {#2}}\atop\raise9pt\hbox{$\scriptstyle\tp$} ${$#3$}}
\newcommand{\ttps}[2]{{#1}\raise5pt\hbox{$\lower12pt\hbox{$\scriptstyle\tp$}\atop \lower0pt\hbox{$\tilde\;$}$}\raise4.5pt\hbox{${\scriptstyle{#2}}$}}
\newcommand{\st}[1]{\mbox{${\,\scriptscriptstyle {#1}}\atop\raise5.5pt\hbox{$*$}$}}

\newcommand{\rd}[1]{\mbox{${\,\scriptscriptstyle {#1}}\atop\raise5.5pt\hbox{$\bullet$}$}}
\newcommand{\rt}[1]{\otimes_\chi}
\newcommand{\lt}[1]{\mbox{${\,\scriptscriptstyle {#1}}\atop\raise5.5pt\hbox{$\ltimes$}$}}
\newcommand{\btr}{\raise1.2pt\hbox{$\scriptstyle\blacktriangleright$}\hspace{2pt}}
\newcommand{\btl}{\raise1.2pt\hbox{$\scriptstyle\blacktriangleleft$}\hspace{2pt}}

\newcommand{\lcr}{\raise1.0pt \hbox{${\scriptstyle\rightharpoonup}$}}
\newcommand{\rcr}{\raise1.0pt \hbox{${\scriptstyle\leftharpoonup}$}}

\newcommand{\ttp}{{\lower12pt\hbox{$\tp$}\atop \hbox{$\tilde\;$}}}

\newcommand{\id}{\mathrm{id}}

\newcommand{\Ru}{\mathcal{R}}

\newcommand{\Q}{\mathcal{Q}}

\newcommand{\C}{\mathbb{C}}

\newcommand{\tp}{\otimes}

\newcommand{\U}{U}

\newcommand{\ve}{\varepsilon}
\newcommand{\gm}{\gamma}
\newcommand{\dt}{\delta}

\newcommand{\la}{\lambda}

\newcommand{\End}{\mathrm{End}}

\newcommand{\Tr}{\mathrm{Tr}}

\newcommand{\diag}{\mathrm{diag}}

\newcommand{\g}{\mathfrak{g}}
\renewcommand{\b}{\mathfrak{b}}
\renewcommand{\k}{\mathfrak{k}}
\newcommand{\h}{\mathfrak{h}}

\newcommand{\s}{\mathfrak{s}}
\newcommand{\n}{\mathfrak{n}}

\newcommand{\nn}{\nonumber}

\newcommand{\p}{\mathfrak{p}}
\renewcommand{\l}{\mathfrak{l}}

\newcommand{\al}{\alpha}

\newcommand{\bt}{\beta}

\newcommand{\be}{\begin{eqnarray}}
\newcommand{\ee}{\end{eqnarray}}

\newtheorem{thm}{Theorem}[section]
\newtheorem{propn}[thm]{Proposition}
\newtheorem{lemma}[thm]{Lemma}
\newtheorem{corollary}[thm]{Corollary}

\theoremstyle{definition}
\newtheorem{remark}[thm]{Remark}

\newcount\prg

\newcommand{\parag}{\advance\prg by1 {\noindent\bf\thesection.\the\prg\hspace{6pt}}}

\begin{document}
\title{Quantum sphere $\mathbb{S}^4$ as a non-Levi conjugacy class}
\author{A.~Mudrov \\
\small Department of Mathematics, \\
\small University of Leicester,
\\
\small University Road,
Leicester,    LE1 7RH, UK.\\
\small e-mail: am405@le.ac.uk\\
}
\date{ }

\maketitle
\begin{abstract}
We construct a $U_\hbar\bigl(\s\p(4)\bigr)$-equivariant quantization of the four-dimensional
complex sphere $\mathbb{S}^4$ regarded as a conjugacy class,  $Sp(4)/Sp(2)\times Sp(2)$,
 of a simple complex group with non-Levi isotropy subgroup,
through an operator realization of the quantum polynomial
algebra $\C_\hbar[\mathbb{S}^4]$ on a highest weight module of $U_\hbar\bigl(\s\p(4)\bigr)$.
\end{abstract}
\\
{\small \underline{AMS classification codes}: 17B10, 17B37, 53D55.}

\section{Introduction}
There are two types of closed conjugacy classes in a simple complex algebraic group $G$.
One type consists of classes that are isomorphic to orbits in the adjoint representation
on the Lie algebra $\g$. They are homogeneous spaces of $G$ whose  stabilizer of the initial point
is a Levi subgroup in $G$.
 Our concern is equivariant quantization of classes of second type, i. e. whose isotropy subgroup {\em is not} Levi.
Regarding the classical series, such classes  are present only in the orthogonal and symplectic groups.

The group $G$ supports a (Drinfeld-Sklyanin) Poisson bivector field $\pi_0\in \Lambda^2(G)$ associated with a solution of the classical Yang-Baxter
equation. This structure makes $G$ a Poisson group, whose multiplication $G\times G\to G$ is a Poisson  map
(here $G\times G$ is equipped with the Poisson structure of Cartesian product).
The Drinfeld-Sklyanin bracket gives rise to the quantum group $U_\hbar(\g)$, which is a deformation,
along the parameter $\hbar$, of the universal enveloping algebra $U(\g)$ in the class of Hopf algebras, \cite{D}.

There is a Poisson structure $\pi_1\in \Lambda^2(G)$  compatible
with the conjugacy action of the Poisson group on itself, \cite{S}.
It means that action map from the Cartesian product of  $(G,\pi_0)$ and $(G,\pi_1)$
to $(G,\pi_1)$ is Poisson. Then $G$ is said to be a Poisson space
over the Poisson group $G$, under the conjugacy action

The Poisson bivector field  $\pi_1$ restricts to every closed conjugacy class making it a Poisson $G$-variety, \cite{AM}. In this sense,
the group $G$ is analogous to $\g\simeq \g^*$ equipped with the canonical $G$-invariant bracket.

Quantization of conjugacy classes with Levi isotropy subgroups has been constructed in various settings, namely,
as a star product and in terms of generators and relations, \cite{EEM,M2}.
Both approaches rely upon the representation theory of the quantum group $U_\hbar(\g)$ and make use of
 the following
facts: a) the universal enveloping algebra  $U(\l)$  of the isotropy subgroup is quantized to a Hopf subalgebra
$U_\hbar(\l)\subset U_\hbar(\g)$,
b) there is a triangular factorization of $U(\g)$ relative to  $U(\l)$, which amounts
to a factorization of quantum groups and  facilitates parabolic induction.
In particular, quantum conjugacy classes of the Levi type have been realized by operators on scalar parabolic Verma modules
 in \cite{M2}.

The above mentioned conditions are violated for non-Levi conjugacy classes, which makes the conventional
methods of quantization inapplicable in this case. In this paper, we show how to overcome these obstructions for the
simplest non-Levi  conjugacy class $Sp(4)/Sp(2)\times Sp(2)$. This is the class of symplectic invertible
$4\times 4$-matrices
with eigenvalues $\pm 1$, each of multiplicity $2$. As an affine variety, it
is isomorphic to the four-dimensional complex sphere $\mathbb{S}^4$. Although the quantization of $\mathbb{S}^4$
can be obtained by other methods, e. g. as in \cite{FRT}, we are interested in  $\mathbb{S}^4$
as an illustration of our approach to a  general non-Levi class.

The idea is to find a suitable highest weight $U_\hbar(\g)$-module where the quantum sphere could be
represented by linear operators. We consider an auxiliary parabolic Verma module $\hat M_\la$ as a starting point. For a
special value of
weight $\la$,  the module $\hat M_\la$ has a singular vector generating a submodule in $\hat M_\la$. The quotient $ M_\la$
of $\hat M_\la$ over that submodule is irreducible. The deformation of the polynomial algebra
$\C[\mathbb{S}^4]$ is realized  by a $U_\hbar(\g)$-invariant subalgebra in $\End(M_\la)$.
This also allows us to describe the quantized polynomial algebra $\C_\hbar[\mathbb{S}^4]$ in terms of generators and relations.

Irreducibility of $M_\la$ implies non-degeneracy of the Shapovalov form on it. In the simple case of
$Sp(4)/Sp(2)\times Sp(2)$ this form can be calculated explicitly. This provides a bi-differential
operator relating the multiplication in $\C_\hbar[\mathbb{S}^4]$ to the multiplication in the dual Hopf algebra
$U_\hbar^*(\g)$, as explained in \cite{KST}.

We start from  description of the classical conjugacy class $Sp(4)/Sp(2)\times Sp(2)$ and the Poisson structure on it.
Next we collect the necessary facts about the quantum group $U_\hbar(\g)$. Further we describe the quantization of
the polynomial algebra  $\C_\hbar[G]$ and its properties. After that  we construct the module $M_\la$
and analyze the submodule structure of the tensor product $\C^4\tp M_\la$. This allows us to realize
 $\C_\hbar[\mathbb{S}^4]$
by operators on $M_\la$ and describe it in generators and relations. In conclusion, we calculate the
invariant pairing between $M_\la$ and its dual and discuss the star product on $\C_\hbar[\mathbb{S}^4]$.

\section{The classical conjugacy class $Sp(4)/Sp(2)\times Sp(2)$}
Let $Sp(4)$ denote the complex algebraic group of matrices preserving
the antisymmetric skew-diagonal bilinear form $C_{ij}=\epsilon_{i}\delta_{ij'}$, where $i'=5-i$,
 $(\epsilon_1,\epsilon_2,\epsilon_3,\epsilon_4)=(1,1,-1,-1)$, and $\dt_{ij}$ is the Kronecker symbol.
We are interested in the conjugacy class of symplectic matrices with eigenvalues
$\pm1$ each of multiplicity $2$. It is an $Sp(4)$-orbit with respect to the conjugation action
on itself.
The initial point $A_o$ of the class
and its isotropy subgroup can be taken as
$$
A_o=
\left(
\begin{array}{cccc}
-1&0&0&0\\
0&1&0&0\\
0&0&1&0\\
0&0&0&-1\\
\end{array}
\right),
\quad
Sp(2)\times Sp(2)=
\left(
\begin{array}{cccc}
*&0&0&*\\
0&*&*&0\\
0&*&*&0\\
*&0&0&*\\
\end{array}
\right)\subset Sp(4).
$$
This conjugacy class is a subvariety in $Sp(4)$
defined by the system of equations
\be
ACA^t-C=0, \quad \Tr(A)=0, \quad A^2-1=0,
\label{ideal_class}
\ee
where $1$ in the third equality is the matrix unit.
This is a system of polynomial equations on the matrix coefficients  $A_{ij}$, which can be written in
an alternative way:
$$
A^t+CAC=0,  \quad \Tr(A)=0, \quad A^2-1=0.
$$
The first two equations are linear and allow for the following non-zero entries:
$$
A=
\left(
\begin{array}{cccc}
a&b&y&0\\
c&-a&0&-y\\
z&0&-a&b\\
0&-z&c&a\\
\end{array}
\right).
$$
The quadratic equation is then equivalent to
\be
\label{sphere}
a^2+bc+yz-1=0.
\ee
Thus, the conjugacy class of $A_o$ is isomorphic to the complex sphere
$\mathbb{S}^4$. The ideal generated by the entries of the matrix equations
(\ref{ideal_class}) along with
the zero trace condition  is, in fact, generated by a single irreducible
polynomial and is the defining ideal of the class.

Consider the r-matrix
$$
r=\sum_{i=1}^4(e_{ii}\tp e_{ii}-e_{ii}\tp e_{i'i'})
+2\sum_{i,j=1\atop i>j}^4(e_{ij}\tp e_{ji}-\epsilon_{i}\epsilon_{j}e_{ij}\tp e_{i'j'})\in \s\p(4)\tp \s\p(4)
$$
solving the classical Yang-Baxter equation, \cite{D}. It induces a Drinfeld-Sklyanin bivector field $\pi_0$ on $Sp(4)$ making it a Poisson group, \cite{D}.
We are concerned with the following Poisson structure, $\pi_1$, on $Sp(4)/Sp(2)\times Sp(2)\simeq \mathbb{S}^4$:
\be
\{A_1,A_2\}=\frac{1}{2}(A_2r_{21}A_1-A_1rA_2+A_2A_1r-r_{21}A_1A_2).
\label{poisson_br}
\ee
This equation is understood in $\End(\C^4)\tp \End(\C^4)\tp \C[\mathbb{S}^4]$ and  is a shorthand matrix form
of the system of $n^2\times n^2$ identities defining the Poisson brackets $\{A_{ij}, A_{kl}\}$ of the coordinate
functions. The subscripts indicate the
copy of $\End(\C^4)$ in the tensor square, as usual in the quantum group literature.
Explicitly,  the  brackets of the  generators $a,b,c,y,z\in \C[\mathbb{S}^4]$ read
$$
\{a,b\}=ab,\quad \{a,c\}=-ac, \quad\{a,y\}=ay,\quad\{a,z\}=-az,
$$
$$
\{b,y\}=by,\quad\{b,z\}=-bz,\quad
\{c,y\}=cy,\quad\{c,z\}=-cz,
$$
$$
\{y,z\}=2a^2+2bc,\quad\{b,c\}=2a^2.
$$
This Poisson structure restricts from $Sp(4)$ and makes
$\mathbb{S}^4$ a Poisson manifold under the conjugacy action of the Poisson group $Sp(4)$, \cite{S}.
In can be shown that such a Poisson structure on $\mathbb{S}^4$ is unique.
\section{Quantum group $U_\hbar(\s\p(4)$)}
\label{QG}
Throughout the paper, $\g$ stands for the Lie algebra $\s\p(4)$.
We are looking for quantization of the polynomial algebra $\C[\mathbb{S}^4]$
along the Poisson bracket (\ref{poisson_br}) that is invariant
under an action of the quantized universal enveloping algebra $U_\hbar(\g)$.
In this section we recall the definition of $U_\hbar(\g)$, following  \cite{D}.

The root system of $\g$ is generated by
the simple positive roots $\al$, $\bt$, which are defined  in the orthogonal
basis $\ve_1$, $\ve_2$ as
$$\al=\ve_1-\ve_2,\quad \bt=2\ve_2.$$
The other positive roots are
$\gm=\al+\bt$ and $\dt=2\al+\bt$.
Root vectors and Cartan elements are represented by the matrices
\be
\begin{array}{c}
e_\al=e_{12}-e_{34},\quad e_\bt=e_{23},\quad e_\gm=e_{13}+e_{24},\quad e_\dt=e_{14},
\\
f_\al=e_{21}-e_{43},\quad f_\bt=e_{32},\quad f_\gm=e_{31}+f_{42},\quad f_\dt=e_{41},
\\
h_\al=e_{11}-e_{22}+e_{33}-e_{44},\quad h_\bt=2e_{22}-2e_{33},
\end{array}
\label{C4rep}
\ee
where $\{e_{ij}\}$ is the standard matrix basis.

The quantized universal enveloping algebra (quantum group) $U_\hbar(\g)$
is a $\C[[\hbar]]$-algebra generated by the elements
$e_\al,e_\bt,f_\al, f_\bt, h_\al,h_\bt$
subject
to the  commutator relations
$$
[h_{\al},e_{\al}]=2 e_{\al},
\quad
[h_{\al},f_{\al}]=-2 f_{\al},
\quad
[h_{\bt},e_{\bt}]=4 e_{\bt},
\quad
[h_{\bt},f_{\bt}]=-4 f_{\bt},
$$
$$
[h_{\al},e_{\bt}]=-2 e_{\bt},
\quad
[h_{\al},f_{\bt}]= 2 f_{\bt},
\quad
[h_{\bt},e_{\al}]=- 2 f_{\al},
\quad
[h_{\bt},f_{\al}]= 2 f_{\al},
$$
$$
[e_{\al},f_{\al}]=\frac{q^{h_{\al}}-q^{-h_{\al}}}{q-q^{-1}},
\quad
[e_{\al},f_{\bt}]=0=[e_{\bt},f_{\al}],
\quad
[e_{\bt},f_{\bt}]=\frac{q^{h_{\bt}}-q^{-h_{\bt}}}{q^2-q^{-2}},
$$
plus the Serre relations
$$
e_{\al}^{3}
e_{\bt}
-
(q^2+1+q^{-2})e_{\al}^{2}
e_{\bt}e_{\al}
+
(q^2+1+q^{-2})e_{\al}
e_{\bt}e_{ \al}^{2}
-
e_{\bt}e_{\al}^{3}
=0
,
$$
$$
e_{\bt}^{2}
e_{\al}
-
(q^2+q^{-2})
e_{ \bt}e_{ \al}e_{\bt}
+
e_{\al}e_{\bt}^{2}
=0,
$$
and similar relations for $f_\al$, $f_\bt$.
Here and further on $q=e^\hbar$.

The comultiplication $\Delta$ and antipode $\gm$ are defined on the generators by
$$
\Delta(h)=h\tp 1+1\tp h, \quad \gm(h)=-h, \quad h\in \h,
$$
$$
\Delta(e_\mu)=e_\mu\tp 1+q^{h_\mu}\tp e_\mu, \quad \gm(e_\mu)=-q^{-h_\mu}e_\mu,
\quad \mu=\al,\bt,
$$
$$
\Delta(f_\mu)=f_\mu\tp q^{-h_\mu}+1\tp f_\mu, \quad \gm(f_\mu)=-f_\mu q^{h_\mu},
\quad \mu=\al,\bt.
$$
The counit homomorphism $\ve\colon U_\hbar(\g)\to \C[[\hbar]]$ is nil on the generators.

\begin{remark}
The quantum group $U_\hbar(\g)$ is regarded  as a $\C[[\hbar]]$-algebra, bearing in mind its application to
deformation quantization.
Accordingly, all its modules are understood as free $\C[[\hbar]]$-modules. However, we will suppress the
reference to $\C[[\hbar]]$ in order to
simplify the formulas. For instance, the vector representation of $U_\hbar(\g)$ will be denoted simply as $\C^4$.
The tensor products and linear maps are also understood over $\C[[\hbar]]$.
\end{remark}

Let us introduce higher root vectors $e_\gm, f_\gm, e_\dt, f_\dt \in U_\hbar(\g)$ (the coincidence in the notation for the weight and the antipode should not cause a confusion) by
$$
f_{\gm}=f_{\bt}f_{\al}-q^{-2} f_{\al}f_{\bt},
,\quad
f_{\dt}=f_{\gm}f_{\al}-q^{2}f_{\al}f_{\gm},
$$
$$
e_{\gm}=e_\al e_\bt-q^{2}e_\bt e_\al, \quad
e_{\dt}=e_\al e_\gm-q^{-2}e_\gm e_\al.
$$
Our definition of  $e_\dt, f_\dt$ is different from the usual definition
 $e_\dt=[e_\al,e_\gm]$,  $f_\dt=[f_\gm,f_\al]$, corresponding to $(\al,\gm)=0$, \cite{ChP}. The reason for that
will be clear later on.
The elements $h_\al, h_\bt$ span the Cartan subalgebra $\h$
and generate the Hopf subalgebra $U_\hbar(\h)\subset U_\hbar(\g)$.
The vectors $e_\al,e_\bt$ along with $\h$ generate
the positive Borel  subalgebra $U_\hbar(\b^+)$
in $U_\hbar(\g)$. Similarly, $f_\al,f_\bt$, and $\h$ generate
the negative Borel  subalgebra $U_\hbar(\b^-)$. They are
Hopf subalgebras of $U_\hbar(\g)$.
\begin{lemma}
The root vectors satisfy the relations
$$
e_{ \gm}e_{ \bt}
-
q^{-2}e_{ \bt}e_{ \gm}
=0
, \quad
[e_{ \al},e_{ \dt}]
=0
,\quad [e_{\bt},e_{\dt}]=0
,\quad
[e_{ \gm},e_{ \dt}]
=0,
$$
$$
f_{\bt}f_{\gm}-q^{2}f_{\gm}f_{\bt}
=0
,\quad
[f_{\al},f_{\dt}]=0
,\quad
[f_{\bt},f_{\dt}]=0
,\quad
[f_{\gm},f_{\dt}]=0.
$$
\label{Borel}
\end{lemma}
\begin{proof}
The first two equalities in both  lines are simply a rephrase of the Serre relations
in the new terms. The last equalities follow from the second and third.
Let us check the third equality, say, in the first line:
\be
e_{\bt}e_{\dt}&=&
e_{\bt}(e_{\al}e_{\gm}-q^{-2} e_{\gm} e_{\al})=e_{\bt}e_{\al}e_{\gm}- e_{\gm}e_{\bt} e_{\al}
\nn\\
&=&
q^{-2}e_{\al}e_{\bt}e_{\gm}-q^{-2}e_{\gm}^2- q^{-2}e_{\gm} e_{\al}e_{\bt}+ q^{-2}e_{\gm}^2
=
q^{-2}e_{\al}e_{\bt}e_{\gm}- q^{-2}e_{\gm} e_{\al}e_{\bt}
\nn\\
&=&
e_{\al}e_{\gm}e_{\bt}- q^{-2}e_{\gm} e_{\al}e_{\bt}=e_{\dt}e_{\bt},
\nn
\ee
as required.
\end{proof}
Denote by $U_\hbar(\n^\pm_0)$ the subalgebras generated by, respectively, positive and
negative Chevalley generators. The Borel subalgebras $\U_\hbar(\b^\pm)$ are freely generated by $\U_\hbar(\n^\pm_0)$ over $\U_\hbar(\h)$ with respect to right or left multiplication.
These equalities facilitate the following
\begin{corollary}
The positive (respectively, negative) root vectors generate a Poincar\`{e}-Birkgoff-Witt
basis in $\U_\hbar(\n^\pm_0)$.
\label{PBW}
\end{corollary}
\begin{proof}
The presence of PBW basis in the quantum group is a well  known fact. However,
we use a non-standard definition of the root vectors $e_\dt$, $f_\dt$, therefore the lemma is substantial.
To prove it, say, for $U_\hbar(\n^-_0)$ one should check that the system of monomials $f_\al^a f_\gm^c f_\dt^d f_\bt^b=f_\al^a f_\dt^d f_\gm^c  f_\bt^b$, where
$a$, $b$, $c$, and $d$ are non-negative integers, is linearly independent  and  complete in $\U_\hbar(\n^-)$.
The ordered sequence of the elements $f_\al,f_\dt'=[f_\gm,f_\al],f_\gm,f_\bt$ does generate a PBW
basis, \cite{ChP}. Using this fact  along with Lemma \ref{Borel} relations, one can easily check the statement
via the substitution
$f'_\dt=f_\dt+(q^{2}-1)f_\al f_\gm$.
\end{proof}
\section{The algebra of quantized polynomials on $Sp(4)$}
\label{S_qG}
We adopt the  convention  throughout the paper that  $G$ stands for the
complex algebraic group $Sp(4)$.
The conjugacy class  of our interest is a closed affine variety
in  $G$, and its polynomial ring is a quotient
of the polynomial ring $\C[G]$ by a certain ideal.
Our goal is to obtain an analogous description of the quantum conjugacy class.
To that end, we need to describe the quantum analog of the algebra $\C[G]$ first.

Recall from \cite{J,B} that the image of the universal R-matrix of the quantum group $U_\hbar(\g)$ in the
vector representation is equal, up to
a scalar factor, to
$$
R=\sum_{i,j=1 }^4 q^{\delta_{ij}-\delta_{ij'}}e_{ii}\tp e_{jj}
  +
  (q-q^{-1})\sum_{i,j=1 \atop i>j}^4(e_{ij}\tp e_{ji}
- q^{\rho_i-\rho_j}\epsilon_i\epsilon_j
e_{ij}\tp e_{i'j'}),
$$
where $(\rho_1,\rho_2,\rho_3,\rho_4)=(2,1,-1,-2)$.

Denote by $S$ the $U_\hbar(\g)$-invariant operator $PR\in \End(\C^4)\tp \End(\C^4)$,
where $P$ is the ordinary flip of $\C^4\tp \C^4$.
This operator has three invariant projectors to its eigenspaces,
among which there is a one-dimensional projector $\sim
\sum_{i,j=1}^4q^{\rho_i-\rho_j}\epsilon_i\epsilon_j e_{i'j}\tp e_{ij'}
$
to the
 trivial  $U_\hbar(\g)$-submodule, call it $\kappa$.

Denote by  $\C_\hbar[G]$ the associative algebra generated by
the entries of the matrix $K=||k_{ij}||_{i,j=1}^4\in \End(\C^4)\tp \C_\hbar[G]$
modulo the relations
\be
S_{12}K_2S_{12}K_2=K_2S_{12}K_2S_{12}
,\quad K_2S_{12}K_2\kappa=- q^{-5}\kappa=\kappa K_2S_{12}K_2.
\label{ideal_quant_linear}
\ee
These relations are understood in $\End(\C^4)\tp \End(\C^4)\tp \C_\hbar[G]$,
and the indices distinguish the two copies of $\End(\C^4)$, as usual.

The algebra $\C_\hbar[G]$ is an equivariant quantization of $\C[G]$, \cite{RS,FRT}, which is different
from the $RTT$-quantization and is not a Hopf algebra. It carries a $U_\hbar(\g)$-action, which is a deformation of the conjugation $U(\g)$-action
 on $\C[G]$. It admits a  $U_\hbar(\g)$-equivariant algebra
monomorphism to  $U_\hbar(\g)$, where the latter is regarded as the adjoint module.
The monomorphism is implemented by the assignment
$$
K\mapsto (\phi\tp \id)(\Ru_{21}\Ru)=\Q\in \End(\C^4)\tp  U_\hbar(\g),
$$
where $\phi\colon   U_\hbar(\g)\to \End(\C^4)$ is the vector representation and $\Ru$ is the
universal R-matrix of $U_\hbar(\g)$.
The matrix $\Q$ is important for our presentation, and the reader is referred
to \cite{M2} for detailed explanation of its role in quantization and
for  its basic characteristics.

\section{The generalized Verma module $M_\la$}
Denote by $\l$ the Levi subalgebra in $\g=\s\p(4)$ spanned by
$
e_\bt,f_\bt, h_\bt, h_\al.
$
It is a Lie subalgebra of maximal rank, and its semisimple part is
isomorphic to $\s\l(2)\simeq \s\p(2)$. The universal enveloping algebra
$U(\l)$ is quantized as a Hopf subalgebra in $U_\hbar(\g)$.
Denote by $\n^+$ and $\n^-$ the nilpotent subalgebras in $\g$ spanned,
respectively, by $\{e_\al, e_\gm, e_\dt\}$ and  $\{f_\al,f_\gm, f_\dt\}$.
The sum $\l+\n^\pm$ is a parabolic subalgebra  $\p^\pm\subset \g$ whose universal
enveloping algebra
is quantized to a Hopf subalgebra in $U_\hbar(\p^\pm)\subset U_\hbar(\g)$.

Let $U_\hbar(\n^\pm)$ be the subalgebras  in $U_\hbar(\g)$ generated by
the quantum root vectors $\{e_\al, e_\gm, e_\dt\}$ and  $\{f_\al,f_\gm, f_\dt\}$, respectively. The quantum group $U_\hbar(\g)$
is a free $U_\hbar(\n^-)-U_\hbar(\n^+)$-bimodule  generated by
$U_\hbar(\l)$:
\be
\label{triangular_fact}
U_\hbar(\p^-)=U_\hbar(\n^-)U_\hbar(\l),
\quad
U_\hbar(\g)=U_\hbar(\n^-)U_\hbar(\l)U_\hbar(\n^+),
\quad
U_\hbar(\p^+)=U_\hbar(\l)U_\hbar(\n^+).
\ee
The factorizations of $U_\hbar(\p^\pm)$ have the structure of smash product.

Fix  a weight $\la\in \h^*$ orthogonal to $\bt$. It can be regarded as a
one-dimensional representation of $U_\hbar(\l)$,
$$
\la\colon e_\bt, f_\bt, h_\bt\mapsto  0, \quad \la \colon h_\al\mapsto (\al,\la),
$$
which can be extended to a representation of $U_\hbar(\p^+)$ by
$\la\colon e_\al\mapsto  0$. Let $\C_\la$ denote the one-dimensional vector space supporting this representation.

Consider the scalar parabolic Verma module $\hat M_\la$ induced from $\C_\la$,
$$
\hat M_\la= U_\hbar(\g)\tp_{U_\hbar(\p^+)}\C_\la.
$$
As a module over $U_\hbar(\n^-)$, it is freely generated by its highest weight
vector $v_\la$. As a module over the Cartan subalgebra, it isomorphic
to $U_\hbar(\n^-)\tp \C_\la$, where $U_\hbar(\n^-)$ is the natural module over
$U_\hbar(\h)$.

The $U_\hbar(\g)$-module $\hat M_\la$ is irreducible except for special values of
$\la$, when $\hat M_\la$ may contain singular vectors. Recall that a weight vector is called
 singular if it is annihilated by the positive Chevalley generators. Such vectors generate
submodules in $\hat M_\la$, where they carry the highest weight.
 We are looking for such $\la$ that $\hat M_\la$ admits a singular
vector of weight $\la-\dt$. Quotienting out the corresponding submodule yields a module that supports
quantization of $\C[\mathbb{S}^4]$.

\begin{propn}
The module  $\hat M_\la$ admits a singular vector of  weight $\la-\dt$
if and only if
$
q^{2(\al,\la)}=-q^{-2}.
$ Then $f_\dt v_\la$ is the singular vector.
\end{propn}
\begin{proof}
The general expression for the  vector of
weight $\la-\dt$ in $M_\la$ is
$$
\bigl(f_{\al}^2f_{\bt}-(a+b)f_{\al}f_{\bt}f_{\al}+abf_{\bt}f_{\al}^2\bigr)v_\la=
\bigl(-(a+b)f_{\al}f_{\bt}f_{\al}+abf_{\bt}f_{\al}^2\bigr)v_\la,
$$
where $a,b$ are some scalars. For this  vector being singular, we have a system
of two equations on $a,b$ resulted from the action of $e_\bt$ and $e_\al$:
$$
\left\{
\begin{array}{rrr}
\bigl(-(a+b)f_{\al}[e_{\bt},f_{\bt}]f_{\al}+ab[e_{\bt},f_{\bt}]f_{\al}^2\bigr)v_\la&=&0,
\\
\bigl(-(a+b)[e_{\al},f_{\al}]f_{\bt}f_{\al}
+abf_{\bt}[e_{\al},f_{\al}]f_{\al}
+abf_{\bt}f_{\al}[e_{\al},f_{\al}]\bigr)v_\la&=&0.
\end{array}
\right.
$$
The non-zero solution of this system is unique (up to permutation  $a \leftrightarrow b$)
and equal to
$$
q^{2(\al,\la)}=-q^{-2}
,\quad a=q^2,\quad  b=q^{-2},
$$
as required.
Finally, notice that
$
f_{\dt}=f_{\al}^2f_{\bt}-(q^2+q^{-2}) f_{\al}f_{\bt}f_{\al}+f_{\bt}f_{\al}^2.
$
This completes the proof.
\end{proof}
Denote by $M_\la$  the quotient of $\hat M_\la$
by the submodule $U_\hbar(\g)f_{\dt}v_\la$.
By Corollary \ref{PBW}, the vectors $f_\al^k f_\gm^l f_\dt^mv_\la$ for all non-negative integer $k,l,m$
form a basis in $\hat M_\la$.
 Therefore, $M_\la$
is spanned by $f_{\al}^kf_{\gm}^l v_\la$, $k,l\geqslant 0$.
\begin{propn}
The module $M_\la$ is irreducible.
\end{propn}
\begin{proof}
Irreducibility follows from non-degeneracy of the invariant bilinear pairing
of $M_\la$ with its dual, see Section \ref{Shapovalov}. One can also verify
that $M_\la$ has no singular vector. Omitting the details,
the action of the positive Chevalley generators on $M_\la$ is given by
\be
e_\al f_{\al}^k f_{\gm}^m v_\la&=&
q^{(\al,\la)+1}\frac{q^{2k}-q^{-2k}}{(q-q^{-1})^2} f_{\al}^{k-1} f_{\gm}^m v_\la,
\nn
\\
e_\bt f_{\al}^kf_{\gm}^m v_\la&=& \frac{q^{2m}-q^{-2m}}{q^{2}-q^{-2}}f_{\al}^{k+1} f_{\gm}^{m-1} v_\la.
\nn
\ee
Here we assume that $k>0$ in the first line and $m>0$ in the second; otherwise the right hand side is nil.
This immediately implies the absence of singular vectors in $M_\la$.
\end{proof}

\section{The $U_\hbar(\g)$-module $\C^4\tp M_\la$}
The tautological assignment (\ref{C4rep}) defines the four-dimensional irreducible representation of $U(\g)$.
Similar assignment on the quantum Chevalley generators and Cartan elements defines a representation of $U_\hbar(\g)$.
Our next object of interest is the $U_\hbar(\g)$-module $\C^4\tp M_\la$.
In particular, we shall study the decomposition of $\C^4\tp M_\la$ into direct sum of irreducible submodules.

Choose the standard basis $\{w_i\}_{i=1}^4\subset \C^4$ of columns with the only nonzero entry
$1$ in the $i$-TtH place from the top. Their weights are $\ve_1, \ve_2,-\ve_2,-\ve_1$, respectively.
As a $U_\hbar(\l)$-module, $\C^4$ splits into the sum of two one-dimensional blocks of weights
$
\pm\ve_1
$
and one two-dimensional block of highest weights
$
\ve_2
$.
The parabolic Verma module contains three blocks
of highest weights
$
\ve_1+\la, \ve_2+\la, -\ve_1+\la,
$
which we denote by $\hat V_{\ve_1+\la}$, $\hat V_{\ve_2+\la}$, $\hat V_{-\ve_1+\la}$.
For generic $\la$ these submodules are
 irreducible, and
\be
\C^4\tp \hat M_\la =\hat V_{\ve_1+\la}\oplus \hat V_{\ve_2+\la}\oplus \hat V_{-\ve_1+\la}.
\label{Verma_split}
\ee
All these blocks are parabolic Verma modules corresponding to the $U_\hbar(\l)$-submodules of $\C^4$.

Clearly $\la+\ve_1$ is the highest weight of $\C^4\tp \hat M_\la$ and $w_1\tp v_\la$ is the
highest weight vector.
The other singular vectors in $\C^4\tp \hat M_\la$ are given next.
\begin{lemma}
The  vectors
\be
u_{\ve_1}&=& w_1\tp v_\la,\nn\\
u_{\ve_2}&=&w_1\tp f_\al v_\la-q\frac{q^{(\al,\la)}-q^{-(\al,\la)}}{q-q^{-1}} w_2\tp v_\la,\nn\nn\\
u_{-\ve_1}&=&  f_\dt w_1\tp  v_\la  +
(q^{(\la,\al)+1}+q^{-(\la,\al)-1})\times\nn\\
&\times&
\Bigl(q w_2\tp f_\bt f_\al v_\la-q^3 w_3\tp f_\al v_\la-q^4\frac{q^{(\la,\al)}-q^{-(\la,\al)}}{q-q^{-1}}w_4\tp v_\la\Bigr)
\nn
\ee
are  singular and generate the submodules
$\hat V_{\ve_1+\la}$, $\hat V_{\ve_2+\la}$, $\hat V_{-\ve_1+\la}$, respectively.
\end{lemma}
\begin{proof}
One should check that $u_{\ve_1}$, $u_{\ve_2}$, $u_{-\ve_1}$ are annihilated by $e_\al$ and $e_\bt$. That is obvious for $u_{\ve_1}$ and relatively easy for $u_{\ve_2}$. The case of $u_{-\ve_1}$ requires
bulky but straightforward calculation, which is omitted here.
\end{proof}
We denote by $V_{\ve_1+\la}$, $V_{\ve_2+\la}$, $V_{-\ve_1+\la}$ the images of
$\hat V_{\ve_1+\la}$, $\hat V_{\ve_2+\la}$, $\hat V_{-\ve_1+\la}$ under the projection
$\C^4\tp \hat M_\la\to \C^4\tp M_\la$, assuming $q^{2(\al,\la)}=-q^{-2}$.
An important fact is that for  $q^{2(\al,\la)}=-q^{-2}$ the
singular vector $u_{-\ve_1}$ turns into $w_1\tp f_\dt v_\la$ and thus
disappears from $\C^4\tp M_\la$. The submodule  $\hat V_{-\ve_1+\la}$ is killed
by the projection $\C^4\tp \hat M_\la\to \C^4\tp M_\la$, so $V_{-\ve_1+\la}=\{0\}$.
\begin{propn}
\label{dir_sum}
The module
$\C^4\tp M_\la$ is a direct sum of the submodules $V_{\ve_1+\la}$ and $V_{\ve_2+\la}$.
\end{propn}
\begin{proof}
The modules $V_{\ve_1+\la}$ and $V_{\ve_2+\la}$ have zero intersection, as they carry different eigenvalues of
the invariant matrix $\Q$, see below.
We must show that the sum $V_{\ve_1+\la}\oplus V_{\ve_2+\la}$ exhausts all of
 $\C^4\tp M_\la$. To that end,
it is sufficient to show that $\C^4\tp v_\la$ lies in  $V=V_{\ve_1+\la}\oplus V_{\ve_2+\la}$.
Indeed, then for all $u\in  U_\hbar(\g)$ and $w\in \C^4$,
$$
w\tp uv_\la=\Delta(u^{(2)})\bigl(\gm^{-1}(u^{(1)})w\tp v\bigr)\in V,
$$
as required.

In what follows $\equiv $ will mean equality modulo $V$.
Obviously, $w_1\tp v_\la\equiv 0$.
Applying $f_{\al}$ to $w_1\tp v_\la$ gives $w_1\tp f_{\al}v_\la+ q^{-(\al,\la)}w_2\tp v_\la\equiv 0$.
Comparing this with $u_{\ve_2+\la}\in V$ we conclude that $w_2\tp v_\la\equiv 0$.
Applying $f_{\bt}$ to $w_2\tp v_\la$ gives $w_3\tp v_\la \equiv 0 $.

Thus, we are left to check that $w_4\tp v\in V$. We have
$$
0\equiv f_{\al}(w_1\tp v_\la)\equiv w_1 \tp  f_\al v_\la ,
\quad
0\equiv f_{\al}( w_2\tp v_\la)= w_2 \tp f_\al v_\la,
$$
$$
0\equiv f_{\al}^2(w_1\tp v_\la)\equiv f_\al(w_1\tp f_\al v_\la)
=w_1\tp f_\al^2 v_\la+q^{-2-(\al,\la)}w_2\tp f_\al v_\la \equiv w_1\tp f_\al^2 v_\la,
$$
\be
0\equiv f_\bt f_{\al}^2(w_1\tp v_\la)\equiv f_\bt(w_1\tp f_\al^2 v_\la)= w_1\tp f_\bt f_\al^2 v_\la.
\label{aux1}
\ee
Further,
$$
0\equiv f_\bt f_{\al}(w_1\tp v_\la)\equiv f_\bt( w_1\tp f_\al v_\la)=
 w_1\tp f_\bt f_\al v_\la,
$$
\be
0\equiv f_\al f_\bt f_{\al}(w_1\tp v_\la)\equiv
f_\al (w_1\tp f_\bt f_\al v_\la)=
 w_1\tp f_\al f_\bt f_\al v_\la
+q^{-(\al,\la)}w_2\tp f_\bt f_\al v_\la.
\label{aux2}
\ee
Combining (\ref{aux1}) and (\ref{aux2}),  we calculate $f_\dt (w_1\tp v_\la)\in V$:
$$
0\equiv
\bigl(-(q^2+q^{-2}) f_{\al}f_{\bt}f_{\al}+f_{\bt}f_{\al}^2\bigr)(w_1\tp v_\la)
\equiv w_1\tp f_\dt v_\la-(q^2+q^{-2})q^{-(\al,\la)}w_2\tp f_\bt f_\al v_\la.
$$
The first equality takes place because $f_\bt(w_1\tp v_\la)=0$.
Therefore $w_2\tp f_\bt f_\al v_\la\equiv 0$, and
$$
0\equiv f_{\bt}(w_2\tp f_\al v_\la)=
w_2\tp f_{\bt}f_\al v_\la+q^2w_3\tp f_\al v_\la\equiv q^2 w_3\tp f_\al v_\la.
$$
Finally,
$$
0\equiv f_\al (w_3\tp v_\la)=w_3\tp f_\al v_\la
-
q^{-(\al,\la)} w_4\tp v_\la\equiv
-
q^{-(\al,\la)}w_4\tp  v_\la,
$$
as required.
\end{proof}

Now consider the action of the matrix $\Q$ on $\C^4\tp \hat M_\la$.
It satisfies a cubic polynomial equation, and its eigenvalues in $\C^4\tp \hat M_\la$
can be found in  \cite{M2}:
$$
q^{2(\la,\ve_1)}=-q^{-2},
$$
$$
q^{2(\la+\rho,\ve_2)-2(\rho,\nu)}=q^{2(\rho,\ve_2)-2(\rho,\nu)}=q^{-2},
$$
$$
q^{2(\la+\rho,-\ve_1)-2(\rho,\nu)}=q^{-2(\la,\ve_1)-4(\rho,\nu)}=-q^{-6}.
$$
The operator $\Q$ is semisimple on $\C^4\tp \hat M_\la$ for generic $\la$.
Due to Proposition \ref{dir_sum}, it is semisimple
on $\C^4\tp M_\la$ and has eigenvalues $\pm q^{-2}$.

\section{Quantization of  $\mathbb{S}^4$}
Let $\phi$ denote  the representation homomorphism $U_\hbar(\g)\to \End(\C^4)$.
The $q$-trace of $\Q$ is a weighted trace $\Tr_q(\Q)=\Tr(D\Q)$,
where $D$ is the diagonal matrix $\diag(q^4,q^2,q^{-2},q^{-4})$. It belongs
to the center of $U_\hbar(\g)$ and hence the center of
$\C_\hbar[G]\subset  U_\hbar(\g)$.
A module of highest weight $\la$ defines a central character $\chi_\la$ of the
algebra $\C_\hbar[G]$, which returns zero on $\Tr_q(\Q)$:
\be
\chi_\la\bigl(\Tr_q(\Q)\bigr)&=&
\Tr\bigl(\phi(q^{h_\la+h_\rho})\bigr)=q^{2(\la+\rho,\ve_1)}
+q^{2(\la+\rho,\ve_2)}
+q^{2(\la+\rho,-\ve_2)}
+q^{2(\la+\rho,-\ve_1)}
\nn\\
&=&q^{2(\la,\ve_1)+4}
+q^{2}
+q^{-2}
+q^{-2(\la,\ve_1)-4}
=
-q^{2}
+q^{2}
+q^{-2}
-q^{-2}
=0,
\nn
\ee
cf. \cite{M2}.
Thus, the $q$-trace of the matrix $\Q$ vanishes in $M_\la$.
Also, the entries of the matrix
$
\Q^2-q^{-4}
$
are annihilated  in
$\End(M_{\la})$, as discussed in the previous section.

\begin{propn}
The image of $\C_\hbar[G]$ in $\End(M_{\la})$ is a quantization of $\C_\hbar[\mathbb{S}^4]$.
It is isomorphic to the subalgebra in $U_\hbar(\g)$ generated by the entries of the matrix $\Q=(\phi\tp \id)(\Ru_{21}\Ru)$, modulo
the relations
\be
\Q^2=q^{-4}, \quad \Tr_q(\Q)=0.
\label{q-rel}
\ee
\end{propn}
\begin{proof}
The center of $\C_\hbar[G]$ is formed by $U_\hbar(\g)$-invariants, which are also central in $U_\hbar(\g)$. Therefore, $\ker \chi_\la$ lies in the kernel of the
representation $\C_\hbar[G]\to \End(M_{\la})$. The quotient of $\C_\hbar[G]$ by the ideal generated by $\ker \chi_\la$
is free over $\C[[\hbar]]$ and is a direct sum of isotypical $U_\hbar(\g)$-components of finite multiplicities, \cite{M1}.
Therefore, the
image of $\C_\hbar[G]$ in $\End(M_{\la})$ is a direct sum of isotypical $U_\hbar(\g)$-components
which are free and finite over $\C[[\hbar]]$.

The ideal in $\C_\hbar[G]$ generated by (\ref{q-rel}) lies in the kernel of
the homomorphism $\phi\colon\C_\hbar[G]\to \End(M_\la)$
and turns into the defining ideal of $\mathbb{S}^4$  modulo $\hbar$.
Therefore this ideal coincides with $\ker \phi$, and the quotient of  $\C_\hbar[G]$
by this ideal is a quantization of $\C[\mathbb{S}^4]$, see \cite{M2} for details.
\end{proof}
We will give a
more explicit description of $\C_\hbar[\mathbb{S}^4]$.
The matrix $\Q$ is the image of the matrix $K$ from Section \ref{S_qG} under the embedding $\C_\hbar[G]\to  U_\hbar(\g)$.
The algebra $\C_\hbar[\mathbb{S}^4]$ is generated by
elements $a,b,c,y,z$ arranged in the matrix
$$
\left(
\begin{array}{cccc}
a&b&y&0\\
c&-q^2a&0&-y\\
z&0&-q^2a&q^2b\\
0&-z&q^2c&q^4a\\
\end{array}
\right).
$$
This matrix is obtained from  $K$
by imposing the linear relations on its entries
derived from (\ref{ideal_quant_linear}) by the substitution $K^2=q^{-4}$, $\Tr_q(K)=0$.
The generators of $\C_\hbar[\mathbb{S}^4]$ obey the relations
$$
ab=q^2ba,\quad ac=q^{-2}ca, \quad ay=q^2ya,\quad az=q^{-2}za,
$$
$$
by=q^2yb,\quad bz=q^{-2}zb,\quad
cy=q^{2}yc,\quad cz=q^{-2}zc,
$$
$$
[b,c]=(q^4-1)a^2,
\quad [y,z]=(q^4-1)a^2+(q^4-1)bc,
$$
plus
$$
a^2+bc+yz=q^{-4},
$$
which is a deformation of  (\ref{sphere}).

Remark that $\C_\hbar[\mathbb{S}^4]$  has a 1-dimensional representation $a,b,c\mapsto 0$, $y,z\mapsto q^{-2}$.
Therefore it can be realized as a subalgebra in the Hopf algebra dual to $U_\hbar(\g)$, as explained in \cite{DM}.

\section{On invariant star product on $\mathbb{S}^4$}
\label{Shapovalov}
It follows from \cite{KST} that the star product on the conjugacy class $Sp(4)/Sp(2)\times Sp(2)$
can be calculated by means of the invariant pairing between the
modules $M_{-\la}^-$ and $M^+_\la$, where $M^+_\la=M_\la$ and
$M_{-\la}^-$ is its restricted dual. The module $M_{-\la}^-$ is the quotient
of the lower parabolic Verma module
$
\hat M^-_{-\la}= U_\hbar(\g)\tp_{ U_\hbar(\p^-)}\C_{-\la}
$
by the submodule $U_\hbar(\g)e_\dt v_{-\la}$.
Explicitly, the pairing is given by the assignment
$$
xv_{-\la}\tp yv_{\la}\mapsto \langle xv_{-\la},yv_{\la}\rangle= \la([\gm(x)y]_{\l}).
$$
Here $x\mapsto [x]_{\l}$ is the projection $U_\hbar(\g)\to  U_\hbar(\l)$
along $U_\hbar'(\n^-)U_\hbar(\g)+U_\hbar(\g)U_\hbar'(\n^+)$, where the prime designates the kernel of the counit.
This projection is facilitated by the triangular factorization (\ref{triangular_fact}) and it
 is a homomorphism of $U_\hbar(\l)$-bimodules.

The modules  $M_{\pm\la}^\pm$ are irreducible if and only if  the pairing is non-degenerate, \cite{Ja}.
Our next goal is to calculate it explicitly.
Put
$$
 x_1=e_\al, \quad x_2=e_{\gm}
,\quad
 \tilde x_1=e_\al, \quad \tilde x_2=q^4e_{\bt}e_{\al}-q^{2} e_{\al}e_{\bt},
\quad  y_1=f_\al, \quad y_2=f_{\gm}.
$$
The twiddled root vectors are related to non-twiddled via the antipode:
$$
\gm(\tilde e_{\gm})= q^4 q^{-h_\bt}e_\bt q^{-h_\al}e_\al-q^{2}q^{-h_\al}e_\al q^{-h_\bt}e_\bt=-q^{-h_\gm}e_\gm,
$$
with a similar relation for the root $\al$.

The following system of monomials constitute
bases in  $M_{-\la}^-$ and $M^+_\la$:
$$
\bigl(y^k_1 y^m_2v_\la\bigr)_{k,m=0}^\infty \subset M^+_\la
,\quad
\bigl(\tilde x^k_1 \tilde x^m_2v_\la\bigr)_{k,m=0}^\infty \subset M_{-\la}^-.
$$
Further we need the identities
\be
[e_\al, f_{\gm}]=-(q+q^{-1})q^{-2}f_{\bt}q^{-h_{\al}},\quad [e_\bt, f_{\gm}]=f_{\al}q^{h_{\bt}},
\\[1pt]
 [e_\gm, f_\al]=-(q+q^{-1})e_\bt q^{h_\al},
 \quad
  [e_\gm, f_\bt]=q^2e_\bt q^{-h_\bt},
\ee
which can be proved directly from the defining relations $U_\hbar(\g)$ and the definition of
$e_\gm$ and $f_\gm$. Also, one can check that
\be
 [e_{\gm}, f_{\gm}]=\frac{q^{h_{\gm}}-q^{-h_{\gm}}}{q-q^{-1}}.
\ee
Thus, for $\nu=\al, \gm$ and any positive integer $k$ we have
\be
 [e_{\nu}, f_{\nu}^k]= q^{h_\nu+1}\frac{1-q^{-2k}}{(q-q^{-1})^2}+q^{-h_\nu-1}\frac{1-q^{2k}}{(q-q^{-1})^2}
.
\ee
\begin{lemma}
The matrix coefficient
$
\langle \tilde x_1^i\tilde x_2^jv_{-\la}, y_1^k,y_2^mv_{\la}\rangle
$
is zero unless $i=k$, $j=m$.
\end{lemma}
\begin{proof}
It follows that $[e_\al,f_\gm^k]$ belongs to the
left ideal $U_\hbar(\g)f_\bt$, hence $x_1 y_2^kv_\la=0$.
We have
$$
x_1 y_1^ky_2^m v_\la =e_\al f_\al^kf_\gm^m v_\la\sim
f_\al^{k-1}f_\gm^m v_\la+f_\al^{k-1}[e_\al,f_\gm^m] v_\la=f_\al^{k-1}f_\gm^m v_\la=y_1^{k-1}y_2^m v_\la.
$$
Using this, we find
$$
\langle \tilde x_1^i\tilde x_2^jv_{-\la}, y_1^k,y_2^mv_{\la}\rangle
=
\langle v_{-\la}, x_2^j x_1^iy_1^ky_2^mv_{\la}\rangle
\sim
\langle v_{-\la}, x_2^j x_1^{i-k}y_2^mv_{\la}\rangle   =0,
$$
assuming $i>k$.
If $i<k$, then $y_1$ is pulled to the left in a similar way, so
the matrix coefficient can be non-zero only if $i=k$.
Suppose $i=k$ but $j>m$.
Then
$$
\langle \tilde x_1^k\tilde x_2^jv_{-\la}, y_1^k,y_2^mv_{\la}\rangle
\sim
\langle v_{-\la}, x_2^k x_1^iy_1^ky_2^mv_{\la}\rangle
\sim
\langle v_{-\la}, x_2^j y_2^mv_{\la}\rangle \sim
\langle v_{-\la}, x_2^{j-m}v_{\la}\rangle =0.
$$
The case $i=k$, $j<m$ is verified similarly.
\end{proof}
Define $[2l]_q!!=\prod_{i=1}^{l}[2l]_q$ for all positive integer $l$ and put $[0]_q!!=1$.
\begin{propn}
The matrix coefficients of the invariant pairing are given by the formula
$$
\langle \tilde x_1^k\tilde x_2^mv_{-\la}, y_1^k     y_2^mv_{\la}\rangle=
q^{m(m-2)+k(k-2)}
\Bigl(\frac{1}{q-q^{-1}}\Bigr)^{k+m}[2k]_q!![2m]_q!!,
$$
for all $k,m=0,1,\ldots$.
\end{propn}
\begin{proof}
The matrix coefficient $\langle \tilde x_1^k\tilde x_2^mv_{-\la}, y_1^k     y_2^mv_{\la}\rangle$ is equal to
$$
\langle v_{-\la},(-q^{-h_\gm}e_\gm)^m (-q^{-h_\al}e_\al)^kf_\al^k,f_\gm^mv_{\la}\rangle=
c\langle v_{-\la},e_\gm^m e_\al^kf_\al^kf_\gm^mv_{\la}\rangle,
$$
where $c=(-1)^{k+m}q^{(k+m)(\al,\la)+m(m-1)+k(k-1)}$.
Further,
$
\langle v_{-\la}, e_\gm^m e_\al^k f_\al^k f_\gm^m v_{\la}\rangle
$
is found to be
$$
\langle v_{-\la}, e_\gm^m e_\al^{k-1}
f^{k-1}_{\al}\Bigl( q^{h_\al+1}\frac{1-q^{-2k}}{(q-q^{-1})^2}+q^{-h_\al-1}\frac{1-q^{2k}}{(q-q^{-1})^2}\Bigr) f_\gm^m v_{\la}\rangle+\ldots
$$
The omitted term is zero, as it involves   $e_\al f_\gm^m v_{\la}=0$.
We continue in this way and get
$$
[k]_q!\prod_{i=1}^k[(\al,\la)+1-i]_q[m]_q!\prod_{i=1}^m[(\al,\la)+1-i]_q
=[2k]_q!![2m]_q!!
\Bigl(\frac{q^{(\al,\la)}q}{q-q^{-1}}\Bigr)^{k+m},
$$
since
$
[(\al,\la)+1-i]_q=\frac{q^{(\al,\la)+1-i}-q^{-(\al,\la)-1+i}}{q-q^{-1}}
=q^{(\al,\la)+1}\frac{q^{-i}+q^{i}}{q-q^{-1}}
=\frac{q^{(\al,\la)+1}[2i]_q}{(q-q^{-1})[i]_q}
    $.
Combining this result with the multiplier $c$ calculated earlier and taking into account $q^{2(\al,\la)}=-q^{-2}$
we prove the statement.
\end{proof}

Let $\C_\hbar[G_{DS}]$ denote the affine coordinate ring on the quantum
group $Sp_q(4)$, i.e. the quantization of the algebra of polynomial functions
on $G$ along the Drinfeld-Sklyanin bracket. It is the Hopf dual to the quantized universal enveloping algebra
$U_\hbar(\g)$, and the reader should not confuse it with $\C_\hbar[G]$ defined
in Section \ref{S_qG}. It is known that the multiplication  in $\C_\hbar[G]$, call it
$\cdot_\hbar\>$,
is a star product, \cite{T}.
Denote by $\C_\hbar[G_{DS}]^\l$ the subalgebra
of $U_\hbar(\l)$-invariants in $\C_\hbar[G]$ under the left co-regular action.
It is a natural right $U_\hbar(\g)$-module algebra under the right co-regular action and is
also a quantization a closed conjugacy class (quotient by the  Levi subgroup).

The  Shapovalov form on $\hat M_\la$ is invertible for non-special $\la$ and its inverse
provides an associative multiplication on $\C_\hbar[G_{DS}]^\l$. For the special value of
$\la$, it has a pole, while its regular part still defines an associative multiplication
on a certain subspace of $\C_\hbar[G_{DS}]$, as argued in \cite{KST}. Description of this subspace
boils down to description of the "stabilizer" of the quantum conjugacy class. It should
be an appropriate deformation of $U(\k)\supset U(\l)$ within $U_\hbar\bigl(\s\p(4)\bigr)$,
where  $\k=\s\p(2)\oplus \s\p(2)$. The algebra  $U_\hbar(\k)$ is unknown to us at present and will be a subject of our future research.

The stabilizer $U_\hbar(\k)$ will determine the subspace $\C_\hbar[G_{DS}]^\k$
of its invariants that supports the associative multiplication
\be
f\tp g\mapsto f*_\hbar g=\sum_{k,m=0}^\infty
q^{-m(m-2)-k(k-2)}\frac{(q-q^{-1})^{k+m}}{[2k]_q!![2m]_q!!}
\bigl(\tilde x_1^{k}\tilde x_{2}^{m}f\bigr)\cdot_\hbar\bigl(y_1^{k} y_{2}^{m}g\bigr),
\label{star}
\ee
for $f,g\in \C_\hbar[G_{DS}]^\k$. This multiplication will be a star-product
on $\C[\mathbb{S}^4]$.

Remark that product (\ref{star}) is not perfectly explicit
because it is expressed through $\cdot_\hbar$,
whose explicit expression through the classical multiplication in $\C[G]$ is unknown.
Also, the new multiplication should be isomorphic to $\cdot_\hbar$,
because $\mathbb{S}^4$ has a unique structure of
Poisson manifold over the Poisson group $G$.
Thus, neither (\ref{star}) nor (\ref{q-rel}) are particularly new with regard to {\em abstract} quantization.
For instance, one can apply the method of characters (which is doable in this special case)
and realize the quantized polynomial algebra on $\mathbb{S}^4$ both as a quotient of $\C_\hbar[G]$
and a subalgebra in $\C_\hbar[G_{DS}]$, \cite{DM}. Alternatively, one can quantize $\mathbb{S}^4$
through the quantum plane, along the lines of \cite{FRT}.
The novelty of the present work is a realization of $\C[\mathbb{S}^4]$ by operators on a highest weight module.
This approach admits far reaching generalization that unites the
non-Levi conjugacy classes and the classes
with Levi isotropy subgroups in a common quantization context.
The general approach to quantization of non-Levi conjugacy classes of simple
complex matrix groups will be the subject of our future research.

\vspace{20pt}
{\bf Acknowledgements}.
This research is supported in part by the RFBR grant 09-01-00504. The author is grateful
to the Max-Plank Institute for Mathematics for hospitality.

\end{document}